\documentclass[12pt]{amsart}
\usepackage{amssymb,amsmath,amsfonts,amsthm}

\usepackage{graphics}
\usepackage{graphicx}                     
\usepackage{enumerate}                    
\usepackage{syntonly}
\usepackage{ipa}
\usepackage{wrapfig}
\usepackage{epsf}
\usepackage{xypic}

\usepackage[utf8]{inputenc}

\newtheorem{theorem}{Theorem}[section]
\newtheorem{lemma}[theorem]{Lemma}

\newtheorem{coro}[theorem]{Corollary}
\newtheorem*{theorem*}{Theorem} 
\newtheorem*{definition*}{Definition}
\newtheorem*{prop*}{Proposition}
\newtheorem*{corollary*}{Corollary}
\newtheorem{ltheorem}{Theorem} 

\newtheorem{example}[theorem]{Example}

\newtheorem{rem}{Remark}[section]

\newcommand{\R}{\mathbb{R}}
\newcommand{\Z}{\mathbb{Z}}

\newcommand{\N}{\mathbb{N}}

\title{Recoding the Classic Hénon-Devaney Map}
\author{Fernando Lenarduzzi}

\thanks{Lenarduzzi is a Post-Doc at UFSCar}

\date{\today}

\address{Universidade Federal de São Carlos, Jardim Guanabara, 13565905 - São Carlos, SP, Brazil.}
\email{}

\begin{document}

\begin{abstract}
In this work we are going to consider the classical Hénon-Devaney map given by
\begin{eqnarray*}
f: \R^2\setminus \{y=0\} &\rightarrow& \R^2 \\
(x,y) &\mapsto& \left(x+\dfrac{1}{y}, y-\dfrac{1}{y}-x\right)
\end{eqnarray*}
We are going to construct conjugacy to a subshift of finite type, providing a global understanding of the map's behavior.We extend the coding to a more general class of maps that can be seen as a map in a square with a fixed discontinuity. 
\end{abstract}

\maketitle

\section{Introduction}

Let us recall the definition of the Boole's map
\begin{eqnarray*}
B: \R \setminus \{x=0\} &\rightarrow& \R \\
x &\mapsto& x - \dfrac{1}{x}
\end{eqnarray*}
which preserves the Lebesgue measure in the real line. The ergodicity of $B$ was proved in 1973 by Adler and Weiss in \cite{Adler}. Some one-dimensional generalizations of this map were studied by S. Mu\~{n}oz in his Ph.D. Thesis \cite{Munoz}, by turning the asymptote at the ``infinity'' into a repelling point.

H\'{e}non's Generating Families and his approach to the Restricted Three-Body Problem has been studied exhaustively by different areas. The asymptotic behavior to ``truncated solutions'' of the problem, presented in \cite{Henon}, is given by
\begin{eqnarray*}
f: \mathbb{R}^2\setminus \{y=0\} &\rightarrow& \mathbb{R}^2 \setminus \{y=-x\}\\
(x,y) &\mapsto& \left(x+\dfrac{1}{y}, y-\dfrac{1}{y}-x\right)
\end{eqnarray*}

This is known today as the \emph{H\'{e}non-Devaney map}, due the work done by Devaney in his paper \cite{Devaney}, in which he constructed a topological conjugation of $f$ to the Baker Transformation. It is clear the resemblance between $B$ and $f$, and that is the reason why $f$ is considered to be the \emph{two-dimensional version} of Boole's map. It is also easy to see that $f$ preservers the Lebesgue measure in the plane and it is natural to ask about its ergodicity. This was asked by Devaney in his paper in 1981 and yet remains open.

An easy remark is that $Df$ is parabolic with only one invariant direction.

\subsection{Statement of results}

We have the following coding for the Hénon-Devaney Map  

\begin{ltheorem}\label{teoB}
There exists $\Sigma_i, \Sigma_j \varsubsetneq \Sigma$ and a surjective map $h:\mathbb{R}^2 \rightarrow \Sigma_i \times \Sigma_j$ such that the following diagram commutes
\begin{displaymath}
\xymatrix{
{\mathbb{R}^2} \ar[r]^{f} \ar[d]_{h} & {\mathbb{R}^2 \ar[d]^{h}} \\
{\Sigma_i \times \Sigma_j} \ar[r]_{\sigma}  & {\Sigma_i \times \Sigma_j}}
\end{displaymath}
where $$\sigma:\Sigma_i \times \Sigma_j \rightarrow \Sigma_i \times \Sigma_j$$ is the product of the usual shift maps restricted to each one of its respective spaces, that induces a subshift for $\Sigma_i \times \Sigma_j$ and
{\footnotesize
$$\Sigma := \bigcup_{m,n \in \overline{\mathbb{N}}\cup\{0\}} \{(s_{-m}\dots,s_{-1};s_0,s_1,\dots,s_n) ; s_i\in\{ -2,-1,1,2\}, s_{-m}=0 \mbox{ or } s_n=0\}$$}
\end{ltheorem}

Remark that the map $h$ is defined in the whole plane. However, if we consider only points whose full orbit is defined, we get a  continuous surjective map
$$h: \mathbb{R}^2  \setminus \bigcup_{n\in\mathbb{N}}f^{-n}(\{y=0\})\cup\bigcup_{n\in\mathbb{N}}f^{n}(\{y=-x\}) \rightarrow \Sigma_i' \times \Sigma_j' $$
where $ \Sigma_i', \Sigma_j' \subset \{-2,-1,1,2\}^\Z $


The set $\Sigma$ consists of sequences that are either bi-infinity with only $\{-2,-1,1,2\}$, or with finite positive part ending in zero, or finite negative part starting in zero or finite sequence starting and ending with zero.

It is important to stress that the map $h$ is only surjective, however we conjecture that we can get a full homeomorphism. All is left to prove is that you have that each one of the ``squares'' that defines the new coordinates converges to a single point. There is strong evidence we have that because the map is hyperbolic. This is discussed more in the comming paper \cite{Pujals2}, a joint work with Ernique Pujals.

Inspired by \cite{Adler} and \cite{Devaney}, we focus on trying to understand how the images and pre-images of each discontinuity spread throughout the plane. We look at the discontinuities turning them into a new system of coordinates of the plane. The Devaney's Theorem can be stated as

\begin{theorem*}
There exists $\Sigma\varsubsetneq \{-1,0,1\}^\Z$ and $h:\mathbb{R}^2 \setminus ( \cup_{\{n\in\mathbb{N}\}}f^{-n}(\{y=0\})\cup\cup_{\{n\in\mathbb{N}\}}f^{n}(\{y=-x\})) \rightarrow \Sigma$ a surjective map such that the following diagram commutes
\begin{displaymath}
\xymatrix{
{\mathbb{R}^2} \setminus \{y=0\} \ar[r]^{f} \ar[d]_{h} & {\mathbb{R}^2 \ar[d]^{h}} \setminus \{y=-x\} \\
{\Sigma} \ar[r]_{\sigma}  & {\Sigma}}
\end{displaymath}
where $$\sigma:\Sigma \rightarrow \Sigma$$ is the usual shift map restricted to each one of its respective spaces.
\end{theorem*} 

The new coding differs a little bit of the conjugacy induced by Devaney in his paper, that is, in his paper he induces a coding that involves only 3 symbols and do not see \emph{explicitly} at each iterate where in the plane the point is at the exact time $n\in\Z$, which can be stated in the following corollary

\begin{coro}
The $i$-coordinate and the $j$-coordinate code the information of the distance to the $x$-axis and the $y$-axis, respectively.
\end{coro}

The two ``coordinates'' given by the product of shifts describes that dynamics along each one of the axis, that can be seen as the point's position relatively to the pre-image of $\{y=0\}$ and the image of $\{y=-x\}$.

What is really important about these images and pre-images is that they form a dense \emph{lamination} of smooth curves in the plane that are transverse to each other, even though we rely on the algebraic form for describing the properties of the curves. It is these geometrical properties that allow us to construct the semi-conjugacy. 

We can extend these results to a map $f:X\setminus\{y=0\}\rightarrow X$ where $X=[-1,1]^2$ or  the Klein Bottle $X=\mathbb{K}$ and such that the map restricted to the sets $(0,1]\times[-1,1]$ and $[-1,0)\times[-1,1]$ are continuous. The union of these sets is mapped bijectively into the square minus a curve connecting the points $(-1,1)$ and $(1,-1)$ in a way that each $(x,0^-)$ is ``mapped" into $(-1,1)$ and the ``image" of $(x,0^+)$ is $(1,-1)$ where
$$
\left\{
\begin{array}{rcl}
(x,0^-)= \lim_{y\rightarrow 0^-} (x,y)\\
(x,0^+)= \lim_{y\rightarrow 0^+} (x,y)
\end{array}\right.
$$

We require some additional properties to these maps. First, we want that the map preserves the boundary, that is, $f(\partial X)\subset \partial X$ and even more, we want that the corner point $(-1,1)$ is parabolic fixed points in the case of the Klein Bottle and $(-1,1)$ and $(1,-1)$ are repealing fixed points for the square. For the latter, we also want that the remaining corner points are attracting.

In the example bellow, the discontinuity is mapped into $\{y=-x\}$ and we have a representation of the images of the curves $(x_0,y)$, for $y<0$ and a fixed $x_0$.

\begin{figure}[h!]\begin{center}\includegraphics[scale=.65]{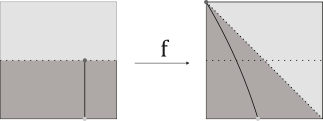}\caption{The map $f$ and the discontinuity.}\end{center}\end{figure}

Here we can state almost the same theorem as before.

\begin{ltheorem}
There exists $\Sigma_i, \Sigma_j \varsubsetneq \Sigma$, the same as in the first theorem, and $h:X \rightarrow \Sigma_i \times \Sigma_j$ a surjective map such that the following diagram commutes
\begin{displaymath}
\xymatrix{
{X}\ar[r]^{f} \ar[d]_{h} & {X \ar[d]^{h}}  \\
{\Sigma_i \times \Sigma_j} \ar[r]_{\sigma}  & {\Sigma_i \times \Sigma_j}}
\end{displaymath}
where $f(\{y=0\})=\partial \overline{f(0,1]\times[-1,1])}\cap \partial \overline{f([-1,0)\times[-1,1])}$ and
$$\sigma:\Sigma_i \times \Sigma_j \rightarrow \Sigma_i \times \Sigma_j$$
is the product of the usual shift maps, which induces a subshift for $\Sigma_i \times \Sigma_j$. If restricted to the set of the continuity points of $f$ then
$$h: X  \setminus \bigcup_{n\in\mathbb{Z}}f^{-n}(\{y=0\})\rightarrow \Sigma_i' \times \Sigma_j' $$
is a semi-conjugacy and $ \Sigma_i', \Sigma_j' \subset \{-2,-1,1,2\}^\Z $.
\end{ltheorem}

The proof follows the same path as before, we use the discontinuities to determine where the dynamics change. There is no need for the discontinuity to be $\{y=0\}$, it can be taken as any simple curve that connects $[-1,1]\times \{-1\}$ and $[-1,1]\times \{1\}$, a graph over the base of the square. The same can be stated to its ``image'', can be any simple curve that connects $\{1\}\times[-1,1] $ and $\{-1\}\times[-1,1]$, a decreasing graph over $\{-1\}\times [-1,1]$.

\section{The Semi-Conjugacy}

We will introduce a system of coordinates that is induced using the pre-images and images of the discontinuities of the H\'{e}non-Devaney map. The coordinates will help us to describe the dynamics in a symbolic way, which will help us to understand exactly the itineraries of each point relative to the position within the curves.

We want to find a conjugacy between the H\'{e}non-Devaney map and a symbolic model that tells us where exactly we are in the plane at each moment. We want to give a description of how the orbits visit a determined region in the plane via the conjugacy of the map and a product of two two-sided subshifts of finite type. 


\subsection{The Discontinuities}

Let us investigate the images of the discontinuities of $f$ and $f^{-1}$, namely the exceptional curves. The main goal is to establish the following result

\begin{lemma}\label{Lemma1}
The set of curves given by $f^{-n} (\{y=0\}), n\in \mathbb{N}$ and $f^{n} (\{y=-x\}), n\in \mathbb{N}$ form two laminations of the plane such that the elements of one is transverse to the elements of the other.
\end{lemma}

\subsubsection{Pre-images of $\{y=0\}$}

Let's take the first step in that direction
$$f^{-1} (\{y=0\}) = \left\{ f^{-1}(t,0); t \in \R \right\} = \left\{\left(t-\dfrac{1}{t}, t\right); t\in \R\right\}$$
which is the Boole's graph with inverted axis.


Before analysing the other cases, let us first understand a few general aspects of $f^{-n}\{y=0\}$ which do not depend on the values of $x$. Denoting $f^{-2}(x,y) = (f^{-2}_x(x,y),f^{-2}_y(x,y))$, we have that each coordinate in $f^{-2}(t,0)$ is an increasing function of the parameter $t$
\begin{eqnarray*}
(f^{-2}_x)'&=& (f^{-1}_x)'+\dfrac{(f^{-1}_x)'+(f^{-1}_y)'}{(f^{-1}_x+f^{-1}_y)^2} \\
&=& 1+\dfrac{1}{t^2}+\dfrac{\left(2+\dfrac{1}{t^2}\right)}{(f^{-1}_x+f^{-1}_y)^2} > 0
\end{eqnarray*}
and it holds for all $t\in\R^\ast\setminus\{t ; f^{-1}_y(t,0)=-f^{-1}_x(t,0)\}$ and, analogously, we have the same for $f^{-2}_y(t,0)$.

Also it is easy to see that $f^{-2}(\{y=0\})\cap f^{-1}(\{y=0\})=\emptyset$ because if there exists $t_1\in\R^\ast$ and $t_2\in\R^\ast\setminus f^{-1}(\{y=-x\})$ such that
$$ \left\{\begin{array}{rcl}
t_1 - \dfrac{1}{t_1} &=& t_2 - \dfrac{1}{t_2} - \dfrac{1}{2t_2-\dfrac{1}{t_2}} \\
t_1 &=& 2t_2 - \dfrac{1}{t_2}
\end{array}\right.$$
that implies that $t_2=0$, which is absurd.

\begin{lemma}\label{Lemma2.3}
Using the notation $f^{-n}(t,0)=(f^{-n}_x(t,0),f^{-n}_y(t,0))$
\begin{itemize}
\item[(a)] Each coordinate of the pre-image $f^{-n}(t,0)$ is an increasing function of the parameter t in each connected component of  $\R\setminus \left(\cup_{j=0}^{n-1} f^{-j}(\{y=-x\}\cap \{y=0\})\right)$;

\item[(b)] $f^{-(n-1)}(t,0)\cap f^{-n}(t,0) = \emptyset$.
\end{itemize}
\end{lemma}

\begin{proof}
Both proofs are made by induction. We already did the induction step before stating the lemma and the proofs resemble the previous cases.

\begin{itemize}
\item[(a)] Assume that it holds for $n$, that is, $(f^{-n}_x)'(t,0)$ and $(f^{-n}_y)'(t,0)$ are positive. Then
$$(f^{-(n+1)}_y)'(t,0)= (f^{-n}_x(t,0)+f^{-n}_y(t,0))' = (f^{-n}_x)'(t,0)+(f^{-n}_y)'(t,0) > 0$$
and
$$(f^{-(n+1)}_x)'(t,0)= (f^{-n}_x)'(t,0) + \dfrac{((f^{-n}_x)'(t,0)+(f^{-n}_y)'(t,0))}{(f^{-n}_x(t,0)+f^{-n}_y(t,0))^2} > 0$$

\item[(b)] If there is $t_1$ and $t_2$ such that
$$ \left\{\begin{array}{rcl}
f^{-n}_x(t_1,0) &=& f^{-(n+1)}_x(t_2,0)= f^{-n}_x(t_2,0) - \dfrac{1}{f^{-n}_x(t_2,0)+f^{-n}_y(t_2,0)} \\
f^{-n}_y(t_1,0) &=& f^{-(n+1)}_y(t_2,0)=f^{-n}_x(t_2,0)+f^{-n}_y(t_2,0)
\end{array}\right.$$
replacing
$$f^{-n}_x(t_2,0) = f^{-n}_y(t_1,0) - f^{-n}_y(t_2,0)$$
from the second equation in the first one, we get
\begin{eqnarray*}
f^{-n}_x(t_2,0)&=& f^{-n}_x(t_1,0) + \dfrac{1}{f^{-n}_y(t_1,0)} \\
&=& f^{-(n-1)}_x(t_1,0) - \dfrac{1}{f^{-n}_y(t_1,0)} + \dfrac{1}{f^{-n}_y(t_1,0)} = f^{-(n-1)}_x(t_1,0)
\end{eqnarray*}

Back to the second equation
$$f^{-(n-1)}_x(t_1,0)+f^{-(n-1)}_y(t_1,0) =: f^{-n}_y(t_1,0)= f^{-(n-1)}_x(t_1,0) + f^{-n}_y(t_2,0)$$
and putting the information together follows that exists $t_1$ and $t_2$ such that
$$ \left\{\begin{array}{rcl}
f^{-(n-1)}_x(t_1,0) &=& f^{-n}_x(t_2,0) \\
f^{-(n-1)}_y(t_1,0) &=& f^{-n}_y(t_2,0)
\end{array}\right.$$
that is $f^{-(n-1)}(t,0)\cap f^{-n}(t,0) \neq \emptyset$, a contradiction.
\end{itemize}
\end{proof}

\begin{rem}\label{Remark2.1}
Using the last lemma, we get 
{\footnotesize
$$f^{-n}(t,0) = \left(f^{-(n-1)}_x(t,0)- \dfrac{1}{f^{-(n-1)}_x(t,0)+f^{-(n-1)}_y(t,0)}, f^{-(n-1)}_x(t,0)+f^{-(n-1)}_y(t,0)\right)$$}
it is easy to see that
\begin{itemize}
\item[(i)] $\lim_{t\rightarrow \infty} f^{-n}_x(t,0)= \infty$;

\item[(ii)] $\lim_{t\rightarrow \infty} f^{-n}_y(t,0)= \infty$.
\end{itemize}
\end{rem}

In the next step, $f^{-2}(\{y=0\})$ will have 4 curves, because of the discontinuities of $f^{-1}$, that is, the pre-images below
\begin{itemize}
\item[(i)] $f^{-1}\left(\left\{\left(t-\dfrac{1}{t}, t\right); t>0 \mbox{ and } t>-\left(t-\dfrac{1}{t}\right)\right\}\right)$

\item[(ii)] $f^{-1}\left(\left\{\left(t-\dfrac{1}{t}, t\right); t>0 \mbox{ and } t<-\left(t-\dfrac{1}{t}\right)\right\}\right)$
\end{itemize}
and the other two cases follow from the previous two because we have that $-f^{-1}(x,y)=f^{-1}(-x,-y)$.

In (i), we get that
$$f^{-2}(\{y=0\}) = \left\{\left(t-\dfrac{1}{t}- \dfrac{1}{2t-\frac{1}{t}}, 2t - \dfrac{1}{t}\right); t>0 \mbox{ and } t>-\left(t-\dfrac{1}{t}\right)\right\}$$

Let $t_1>0$ and $t_2\in\{t\in\R^\ast; t>0 \mbox{ and } f^{-1}_y(t,0)>-f^{-1}_x(t,0)\}$ such that $f^{-1}_x(t_1,0)=0=f^{-2}_x(t_2,0)$ then
$$f^{-2}(t_1,0) = \left(f^{-1}_x(t_1,0)- \dfrac{1}{f^{-1}_x(t_1,0)+f^{-1}_y(t_1,0)}, f^{-1}_x(t_1,0)+f^{-1}_y(t_1,0)\right)$$
and recalling the choice of $t_1$ we have that
$$f^{-2}(t_1,0) = \left(-\dfrac{1}{f^{-1}_y(t_1,0)}, f^{-1}_y(t_1,0)\right)$$

Once $f^{-1}_y(t_1,0)>0=f^{-1}_x(t_1,0)$ and $t_1$ and $t_2$ are in the same connected component due the choice of $t_2$, we conclude from the previous lemma that $t_1<t_2$ and also that
$$f^{-1}_y(t_1,0) < f^{-1}_y(t_2,0)+f^{-1}_x(t_2,0)=f^{-2}_y(t_2,0)$$
because $f^{-2}_y(t_2,0)>0$. This means, using $(b)$ from the previous lemma, that the pre-image of the curve above the discontinuity of the inverse is another curve that is above $f^{-1}(\{y=0\})$ with $f^{-1}_y(t,0)>-f^{-1}_x(t,0)$.

For (ii) we will prove something similar but now the curve is below 0 and above $f^{-1}(\{y=0\})$ with $t<0$. Indeed, first note that for every $t\in\{t\in\R^\ast; t>0 \mbox{ and } f^{-1}_y(t,0)<-f^{-1}_x(t,0)\}$
$$f^{-2}_y(t,0) = f^{-1}_y(t,0)+f^{-1}_x(t,0)<0$$

Here we need to see that this is also above the first pre-image. The easiest way to check directly, take two parameters that have the same $x$ coordinate and see where the curves are. The induction step is very different from the general proof we will see a bit later. 

Let us find $t_0$ one of the parameters of $(ii)$ such that $f^{-2}_x(t_0,0)=0$ and examine what happens in the $y$ coordinate. All we have to do is to find the right solution to
$$t-\dfrac{1}{t} - \dfrac{1}{2t-\dfrac{1}{t}} = 0$$
or equivalently
$$2t^4 - 4t^2+1 =0$$
where we get that $t_0 = \sqrt{1 - \dfrac{\sqrt{2}}{2}}$. Hence
$$0>f^{-2}_y(t_0,0) = 2t_0 - \dfrac{1}{t_0}= \sqrt{4-\sqrt{2}}- \sqrt{2 +\sqrt{2}}>-1$$
that is, this curve is between $\{y=0\}$ and its negative pre-image.



Using the same ideas we introduced before, we will prove that the $f^{-(n+1)}(\{y=0; f^{-j}_y(t,0)>-f^{-j}_x(t,0), \ 0\leq j \leq n\})$ is a family of curves that ``covers'' the upper part of $\R^2$ in the sense that, for each point in this region, there is a curve of the mentioned family above and bellow it. The previous idea will be the induction step and the strategy is the same. Indeed, let us proceed in the same way, let $t_{n+1} \in \{t\in\R; f^{-j}_y(t,0)>-f^{-j}_x(t,0), \ 0\leq j \leq n\}$ such that
$$f^{-(n+1)}_x(t_{n+1},0) = 0$$

We want to see that each curve containing $f^{-n}(t_n,0)$ is an increasing sequence of curves, to do that we use the induction step. Let us assume that it holds for $n$, that is
$$f^{-n}_y(t_n,0) > f^{-(n-1)}_y(t_{n-1},0)>0=f^{-(n-1)}_x(t_{n-1},0)$$
and then $t_n\in \{t\in\R; f^{-j}_y(t,0)>-f^{-j}_x(t,0), \ 0\leq j \leq n\}$. Therefore we can compare $t_n$ and $t_{n+1}$ because they are in the same connected component of $f^{n+1}$. Hence
$$f^{-(n+1)}(t_n,0) = f^{-1}(f^{-n}(t_n,0))=\left(- \dfrac{1}{f^{-n}_y(t_n,0)}, f^{-n}_y(t_n,0)\right)$$ 
and once
$$f^{-(n+1)}_x(t_n,0) = -\dfrac{1}{f^{-n}_y(t_n,0)}< 0 = f^{-(n+1)}_x(t_{n+1},0)$$
we get, by lemma \ref{Lemma2.3}, that $t_n<t_{n+1}$ which implies 
$$f^{-(n+1)}_y(t_{n+1},0) > f^{-n}_y(t_{n},0)$$

The conclusion here is the same as described previously in (i), that is to say the $(n+1)$-th curve is above the $n$-th curve and also that $(f^{-n}_y(t_{n},0))_{n\in\N}$ is increasing. 



\begin{rem}
Notice that the area between two subsequent pre-images of $\{y=0\}$ and in the same side of the anti-diagonal is mapped inside the pre-images of those said curves.

\end{rem}

But we cannot conclude yet what we stated before because we do not know if the curves ``diverge'', we need to understand how the sequence we found behaves:

\begin{lemma}\label{LemmaL}
The sequence $(f^{-n}_y(t_{n},0))_{n\in\N}$ diverges.
\end{lemma}

\begin{proof}
Suppose that there exists
$$\lim_n f^{-n}_y(t_{n},0) = \sup_n f^{-n}_y(t_{n},0) = L>0$$
and because $t_1=1$ we know that $L>1$. Observe now that
$$f(0,L) = \left(\dfrac{1}{L}, L-\dfrac{1}{L}\right)$$
with $L-\dfrac{1}{L}>0$. Also, there exists $n_0\in\N$ such that
$$f^{-n}_y(t_n,0)\geq L-\dfrac{1}{L} \qquad \forall n\geq n_0$$

For all $n\geq n_0$, let $t^L_n$ be the parameter in which
$$f^{-n}_x(t^L_n,0) = \dfrac{1}{L}>0$$
in the same connected component of $t_n$. Thus
$$f^{-n}_y(t^L_n,0) > f^{-n}_y(t_n,0) \geq L - \dfrac{1}{L}$$

This tells us that $f(0,L)$ is in the first quadrant, also it cannot be one of the discontinuities of $f^{-n}$ and it is bellow the curves containing $f^{-n}(t_n,0)$ for all $n\geq n_0$.


Now we know that $f(0,L)$ is in the area delimited by the coordinates axis and bounded above by the the curve containing $f^{-n_0}(t_{n_0},0)$. Recalling the previous remark, we see that $(0,L)= f^{-1}\left(\dfrac{1}{L}, L - \dfrac{1}{L}\right)$ is in the region limited above by $f^{-1}(f^{-n_0}(t_{n_0},0))$ although it is also above it, a contradiction.
\end{proof}

Now we want to understand of the generalization of $(ii)$, that is, the points that in the $n$-th pre-image changes sign with respect the anti-diagonal:
$$D_n:=\{t\in\R; f^{-j}_y(t,0)>-f^{-j}_x(t,0), \ 0\leq j < n \mbox{ and the opposite for }n\}$$

Here the induction hypothesis is that for $n$ is under the curve $\{y=0\}$ and above the pre-image of $D_{n-1}$. Our objective here is to see that the $(n+1)$-th pre-image is trapped between $f^{-1}(D_n)$ and $\{y=0\}$. The idea here is purely geometric: consider a straight line $l$ that comes from the $y$-axis and that touches $D_{n+1}$ as illustrated below. All it is left is to see what happens to the pre-image of this line.


When we look at the pre-image of this line and use the induction hypothesis, also remember that that $D_j$ are sets that are moving up according to the previous case, we observe that the pre-image of $l$ touches each one of the curves only once. This implies that $f^{-1}(D_{n+1})$ is either between $D_n$ and $D_1$, under $D_1$ or it is above $D_n$ as we wanted. 

If the first one occurs, then there would be a line joining $f^{-1}(D_{n+1})$ and $f^{-1}(D_1)$ that does not touch $f^{-1}(D_{n})$ although its would intersect $D_n$ which is a contradiction. 

If the second one holds, something similar to the previous case would happen: the line connecting $f^{-1}(D_{n+1})$ and $f^{-1}(D_1)$ would go trough $f^{-1}(D_1)$ twice, implying that the image of the curves crosses $D_1$ also twice, which is absurd.  Therefore, the third case holds as we wanted.

\begin{coro}\label{insidecurves}
The curves given by $f^{-1}(D_n)$ converges pointwise to 0.
\end{coro}

\begin{proof}
To establish this just observe that, in case it does not, each limit would be a limiting point to the direct image of it, that is, a limit for the curves we already proved that diverge in Lemma \ref{LemmaL}.
\end{proof}

\begin{rem}
It is interesting to see that these points that are in the curves defined by $f^{-1}(D_n)$ will be mapped ``away'' from the $\{y=0\}$, more precisely, as a consequence of Lemma \ref{LemmaL} the curves that contain each $D_n$ covers the lower part of the plane. We will explore a bit more of this in the subsection \ref{Subshift}.
\end{rem}




\subsubsection{Images of $\{y=-x\}$}

We want to make the same study for the images of the inverse map's discontinuity
$$f(\{x=-y\}) = \left\{f(t,-t); t\in\R \right\} = \left\{ \left(t-\dfrac{1}{t}, -2t+\dfrac{1}{t}\right); t\in \R\right\}$$


Observe that, using the previous notation, one can write
$$f(t,-t) = (f^{-1}_x(t,0), - f^{-2}_y(t,0))$$
and that is in fact a little more general as describe in the next remark.

\begin{rem}\label{Remark2.3}
We have the following identification
$$f^n(t,-t)=(f^{-n}_x(t,0), -f^{-(n+1)}_y(t,0))$$
\end{rem}

\begin{proof}
All we need to do is to check it by induction. Assume it holds for $n$, then
$$f^{n+1}(t,-t) = f(f^n(t,-t)) = f(f^{-n}_x(t,0), - f^{-(n+1)}_y(t,0))$$
which implies that
{\footnotesize
$$f^{n+1}(t,-t) = \left(f^{-n}_x(t,0) - \dfrac{1}{f^{-(n+1)}_y(t,0)}, - f^{-(n+1)}_y(t,0) + \dfrac{1}{f^{-(n+1)}_y(t,0)} - f^{-n}_x(t,0) \right)$$}

Recalling the definitions of $f^{-n}_x(t,0)$ and $f^{-n}_y(t,0)$,
$$f^{-(n+1)}_x(t,0)= f^{-n}_x(t,0)- \dfrac{1}{f^{-n}_x(t,0)+f^{-n}_y(t,0)} =  f^{-n}_x(t,0)- \dfrac{1}{f^{-(n+1)}_y(t,0)}$$
and it is straightforward to see that
$$f^{n+1}(t,-t) = \ \left(f^{-(n+1)}_x(t,0), - f^{-(n+2)}_y(t,0) \right)$$
\end{proof}

To understand the behaviour of $f^n(t,-t)$ we need to make a refinement of the remark \ref{Remark2.1}. We started this analysis looking only at the ``infinity'', but we need to take it a bit further.

\begin{rem}
Recalling that the pre-images of $\{y=0\}$ are given by
{\footnotesize
$$f^{-n}(t,0) = \left(f^{-(n-1)}_x(t,0)- \dfrac{1}{f^{-(n-1)}_x(t,0)+f^{-(n-1)}_y(t,0)}, f^{-(n-1)}_x(t,0)+f^{-(n-1)}_y(t,0)\right)$$}
we can understand the full behaviour of $f^{-n}(t,0)$ near the discontinuities as described bellow
\begin{itemize}
\item[(i)] $\lim_{t\rightarrow t_d^-} f^{-n}_x(t,0)= \infty$, where $t_d\in \left(\cup_{j=0}^{n-1} f^{-j}(\{y=-x\}\cap \{y=0\})\right)\cup \{\infty\}$;

\item[(ii)] $\lim_{t\rightarrow t_d^+} f^{-n}_x(t,0)= -\infty$, where $t_d\in \left(\cup_{j=0}^{n-1} f^{-j}(\{y=-x\})\cap \{y=0\}\right)\cup \{-\infty\}$; 

\item[(iii)] $\lim_{t\rightarrow t_d^-} f^{-n}_y(t,0)= \infty$, where $t_d\in f^{-(n-1)}(\{y=-x\}\cap \{y=0\})\cup \{\infty\}$;

\item[(iv)] $\lim_{t\rightarrow t_d^+} f^{-n}_y(t,0)= 0$, where $t_d\in f^{-(n-1)}(\{y=-x\}\cap \{y=0\})$;

\item[(v)] $\lim_{t\rightarrow t_d^-} f^{-n}_y(t,0)= 0$, where $t_d\in \left(\cup_{j=0}^{n-2} f^{-j}(\{y=-x\}\cap \{y=0\})\right)$;

\item[(vi)] $\lim_{t\rightarrow t_d^+} f^{-n}_y(t,0)= -\infty$, where $t_d\in \left(\cup_{j=0}^{n-2} f^{-j}(\{y=-x\}\cap \{y=0\})\right)\cup \{-\infty\}$;
\end{itemize}
\end{rem}

\begin{proof}
Using the recurrence formula we already know, one can deduce by induction that
$$f^{-n}(t,0) = \left(t - \sum_{j=1}^n\dfrac{1}{f^{-j}_y(t,0)}, f^{-(n-1)}_x(t,0)+f^{-(n-1)}_y(t,0)\right)$$

To get items $(i)$ and $(ii)$, all we need to do is understand what the first coordinate of the previous relation is telling us. Just notice that for each $t_d$ only one of the $f^{-(j+1)}_y(t_d,0)=f^{-j}_x(t_d,0)+f^{-j}_y(t_d,0)$ vanishes per time and the sign comes from which side it approaches 0.

The items $(iii)$ and $(iv)$ come directly from the formula and observing that the second coordinate vanishes. For the last two just replace $f^{-(n-1)}_x(x,0)$ by the induction formula we first stated here, that is,
$$f^{-n}_y(t,0)= f^{-(n-1)}_x(t,0)+f^{-(n-1)}_y(t,0) = f^{-(n-1)}_y(t,0) + x - \sum_{j=1}^n\dfrac{1}{f^{-j}_y(t,0)}$$
and repeat the analysis we did in $(i)$.
\end{proof}

It is an interesting observation that for $f^n(t,-t)= (f^{-n}_x(t,0)$,\linebreak $-f^{-(n+1)}_y(t,0))$, the discontinuities of $f^{-n}_x(t,0)$ and $-f^{-(n+1)}_y(t,0)$ are the same, that is, $(-f^{-(n+1)}_y(t,0))$ is an continuous function in each connected component of $\left(\cup_{j=0}^{n-1} f^{-j}(\{y=-x\}\cap \{y=0\})\right)$ that is onto $\R$.

Putting this all together, we can see that what is happening in this case is something very similar to the case of $f^{-1}(t,0)$. It is actually pretty much the same idea but the exceptional curve here that we have to avoid is $\{y=0\}$, that is, when the $n$-th image touchs this curve its $(n+1)$-th image will split into two different curves just like happened before.


The relation of the discontinuities tells us exactly which are the points that nullify $f^n_y(t,-t)$: the discontinuities of $f^{-(n+1)}_x(t,0)$. Hence it is expected to have something similar to lemma \ref{Lemma2.3} and the other results that follow from it.

\begin{lemma}
With the notation $f^{n}(t,-t)=(f^n_x(t,-t),f^n_y(t,-t))$
\begin{itemize}
\item[(a)] $f^{n}_x(t,-t)$ is increasing and $f^n_y(t,-t)$ is decreasing with respect the parameter t in each connected \linebreak component of  $\R\setminus \left(\cup_{j=0}^{n-1} f^{-j}(\{y=-x\}\cap \{y=0\})\right)$;

\item[(b)] $f^{(n-1)}(t,-t)\cap f^{n}(t,-t) = \emptyset$.
\end{itemize}
\end{lemma}

\begin{proof}
The proof becomes very easy when we use the remark \ref{Remark2.3}. The first item follows directly from \ref{Lemma2.3}. The proof of the second one follows the same idea, all we have to see is that we can reduce this to the case (b) of the original lemma.

Suppose there exists $t_1$ and $t_2$ such that
$$ \left\{\begin{array}{rcl}
f^{(n-1)}_x(t_1,-t_1) &=& f^{n}_x(t_2,-t_2) \\
f^{(n-1)}_y(t_1,-t_1) &=& f^{n}_y(t_2,-t_2)
\end{array}\right.$$
implying that
$$ \left\{\begin{array}{rcl}
f^{-(n-1)}_x(t_1,0) &=& f^{-n}_x(t_2,0) \\
f^{-n}_y(t_1,0) &=& f^{-(n+1)}_y(t_2,0)
\end{array}\right.$$

Just using the first equation in the second we get
$$f^{-(n-1)}_y(t_1,0) = f^{-n}_y(t_2,0)$$
in other words
$$f^{-(n-1)}(t,0)\cap f^{-n}(t,0) \neq \emptyset$$
which is a contradiction
\end{proof}

We have to check again the reorganizing pattern of the curves under the action of $f$. The good news is that, disconsidering the change of sign, it is like the previous case: the curves that are the $n$-th images of $\{t\in\R; f^{-j}_y(t,0)>-f^{-j}_x(t,0), \ 0\leq j \leq n\}$ covers the whole part under the anti-diagonal and the images of the other positive parameters go inside the area delimited by $f(\{y=-x\})$, just like happened in the other case. But that is just to look at what we already did and it will follow directly from the relation between that image and the pre-image.

Just obverse that we already know that, in this set 
$$f^{-n}_y(t,0) < f^{-(n+1)}_y(t,0) \qquad \forall n\in \N$$
or in other words
$$f^{(n-1)}_y(t,-t) > f^{n}_y(t,-t)$$
the curves given by these sets are moving down in each step. Using the same analysis one can see that $D_n$ is a curve that is between the curves given by $f(t,-t)$, just like what happens to the pre-images.

\vspace{.3cm}

If we could conclude the density of curves stated in \ref{Lemma1} it would imply that we actually have a conjugacy. As we already mentioned, we conjecture this to be true because we have strong evidence linked to the non-uniform hiperbolic behavior of the map.

\subsection{New Coordinates}

This construction just uses all the information we got so far: how the images and pre-images cover the whole plane. We will use the symmetries of the map to make easier the understanding of the proof. At this first instance we will only consider the upper part of the plane and, once $f(-x,-y)=-f(x,y)$, everything can be mirrored to the lower part of $\R^2$.

Let $\{R_i\}_{i\geq0}$ and $\{L_j\}_{j\geq0}$ be the families of curves given respectively by the lamination of the highest $i$-pre-image of $\{y=0\}$ and the highest $j$-image of $\{y=-x\}$ with respect to the anti-diagonal. Denote the mirrored curves by  $\{R_i\}_{i<0}$ and $\{L_j\}_{j<0}$

The key here is to use the only information we have: the boundaries of each intersection are curves that we know exactly how it moves. Denote by the pair $(i,j)$ the region delimited by $(R_i,R_{i-1}, L_j,L_{j-1})$. Observe now that once the curves that delimit each $(i,j)$ are related by the image and the pre-image of $f$, that is, $R_{i+1}$ is the pre-image of one part of $R_i$ and the same holds for $L_{j+1}$, because it is the image of a part of $f$. 

We need the additional information that the Corollary \ref{insidecurves} gives us: we have to add the curves inside each $(i,j)$. The curves determined in \ref{insidecurves} are curves inside $R_1$ that are induced by the pre-images of the $\{R_i\}_{i<0}$, that is, it is a family of curves $R_{1\oplus i_1}$ contained in $R_1$ such that
$$f(R_{1\oplus i_1})\subset R_{i_1} \qquad i_1\in\mathbb{Z}^-$$
where $R_{1\oplus i_1} = f^{-1}(R_{i_1})$. Using the analogous definition, we can define $\{L_{1\oplus j_1}\}_{j_1\in\mathbb{Z^-}}$.


However, this is not restricted to $R_1$ and $L_1$: it is a consequence of the choice of the pre-images and how they distribute above the plane that we can ``extend'' the curves inside of $R_1$ and $L_1$ to any $R_i$ and $L_j$, for $i,j>0$.

In order to do so, the study of how the pre-images of $f^{-n}(\{y=0\})$ distribute over the plane proved that
\begin{eqnarray*}
f(i,j) = \left\{ \begin{array}{rcl}
& (i-1, j+1)& \ i>1 \mbox{ and } j>0; \\
& (i-1, 1\oplus j)& \ i>1 \mbox{ and } j<0
\end{array}\right.
\end{eqnarray*}
which takes the the subdivisions of $R_1$ and $L_1$ to all the previously mentioned sets. Therefore we can define the sets $\{R_{i_0\oplus i_1}\}$ and $\{L_{j_0\oplus j_1}\}$, for integers of alternating signs and $R_{i_0\oplus 0}:= R_{i_0}$ and $L_{j_0\oplus 0}:= L_{j_0}$.

\begin{figure}[h!]\begin{center}\includegraphics[scale=.28]{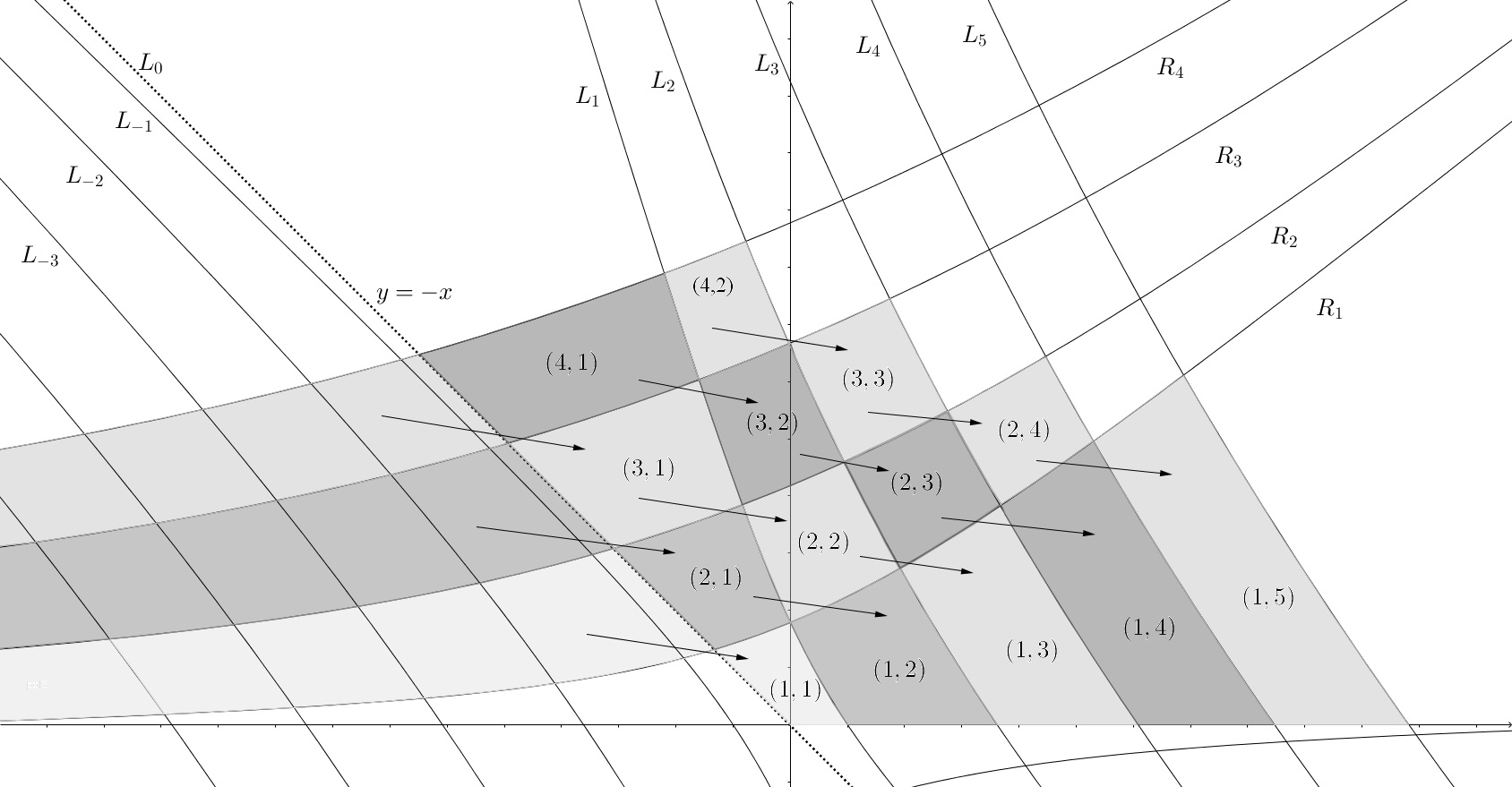}\caption{The dynamics of the part above $R_1$.}\end{center}\end{figure}

Just like above, the pair $(i_0 \oplus i_1,j_0\oplus j_1)$ will denote the region delimited by the curves $(R_{i_0 \oplus i_1},R_{i_0 \oplus i_1 -1 }, L_{j_0\oplus j_1},L_{j_0\oplus j_1-1})$.

Again, we know that $f(R_{1\oplus i_1})\subset R_{i_1}$. This shows us that we can also induce an lamination within the region between $R_{1\oplus i_1}$ and $R_{1\oplus i_1-1}$ that comes from what we defined in the previous step, i.e., the one we already have inside $R_{i_1}$, namely $\{R_{i_1\oplus i_2}\}_{i_2 \in \mathbb{Z^+}}$. With this in hand, we can define the curves $\{R_{1 \oplus i_1 \oplus i_2}\}$ and then extend it to the curves
$$\{R_{i_0\oplus i_1 \oplus i_2}; \mbox{ where } i_0, i_1, i_2\in \mathbb{Z} \mbox{ of alternating signs}\}$$
that lay inside the region delimited by $R_{i_0\oplus i_1}$ and $R_{i_0\oplus i_1 -1}$. Using the same argument we can define the curves inside each region delimited by $L_{j_0\oplus j_1}$ and $L_{j_0\oplus j_1-1}$: the set of curves $\{L_{j_0\oplus j_1 \oplus j_2}\}$.

Like before, it is possible to subdivide the region delimited by each $R_{1 \oplus i_1 \oplus i_2}$ and $R_{1 \oplus i_1 \oplus i_2-1}$ once we know that $f^2(R_{1\oplus i_1\oplus i_2})\subset R_{i_2}$, we can continue to subdivide each region we got in the previous step. Proceeding in the same way we stated before we got, by construction, two dense sets of curves that are transversal:
$$\{R_{i_0 \oplus_{n\in\mathbb{N}} i_n}; \mbox{ where } i_n\in \mathbb{Z} \mbox{ of alternating signs}\}$$
and the correspondent in the other direction
$$\{L_{j_0 \oplus_{m\in\mathbb{N}} j_m}; \mbox{ where } j_m\in \mathbb{Z} \mbox{ of alternating signs}\}$$

Due the density of the curves, if a point does not lie over any of these curves, it may be represented by the new coordinates gives regarding these curves: the intersection of all regions that we introduced, that is, we may identify each point by the coordinates $\left(i_0 \oplus_{n\in\mathbb{N}} i_n, j_0 \oplus_{m\in\mathbb{N}} j_m\right)$. If a point lies over a $R$ or a $L$ curve, then it means that it has a finite representation in that coordinate, e.g., if we have a point that lies over the $R_{i_0\oplus i_1\dots \oplus i_k}$, it will have a finite $i$-coordinate: $(i_0\oplus i_1\oplus\dots \oplus i_k,j_0 \oplus_{m\in\mathbb{N}} j_m)$. The density playing along the transversality give the unique representation of each point in the plane.

Although the coordinates look a bit terrifying, it is a very useful way to describe de dynamics because of the way we constructed them
{\tiny
\begin{eqnarray*}
f\left(i_0 \oplus_{n\in\mathbb{N}} i_n, j_0 \oplus_{m\in\mathbb{N}} j_m\right)=\left\{\begin{array}{rcl}
&\left(i_0 - 1 \oplus_{n\in\mathbb{N}} i_n, j_0 + 1\oplus_{m\in\mathbb{N}} j_m\right)& \ i_0>1 \mbox{ and } j_0>0; \\
& \left(i_1 \oplus_{n\geq 2} i_n, j_0 + 1\oplus_{m\in\mathbb{N}} j_m\right)& \ i_0=1 \mbox{ and } j_0>0; \\
&\left(i_0 - 1 \oplus_{n\in\mathbb{N}} i_n, 1\oplus j_0 \oplus_{m\in\mathbb{N}} j_m\right)& \ i_0>1 \mbox{ and } j_0<0; \\
& \left(i_1 \oplus_{n\geq 2} i_n, 1 \oplus j_0 \oplus_{m\in\mathbb{N}} j_m\right)& \ i_0=1 \mbox{ and } j_0<0;
\end{array}\right.
\end{eqnarray*}}
and using the mirroring property that the H\'enon-Devaney has, one can see what happens with the signs changed. With this in hand, we will proceed to give an complete symbolic description of the map.


\subsection{The Subshift}\label{Subshift}

In this section we will explain how to encrypt the H\'enon-Devaney map into a subshift of finite type. Let \linebreak $\mathcal{A}= \{-2,-1,0,1,2\}$ be the alphabet of symbols we will use but, however, it will not be a complete shift. We want to find a conjugacy between the original map and a product of two subshifts, one for each coordinate we introduced before.

At first we will consider points which have the complete description on terms of the $(i,j)$ coordinates. We will consider them first not only because they will give us the idea behind the coding but also because they form the set in $\mathbb{R}^2$ that the dynamics is defined for all iterations backwards and forwards. Also, in terms of the Lebesgue measure in the plane, the complement, that is, the point which have finite orbit backward or forward, have zero measure. This comes from the fact that these points lay all on the set given by
$$\left(\bigcup_{n\in\mathbb{N}}f^{-n}(\{y=0\})\right) \cup \left(\bigcup_{n\in\mathbb{N}}f^{n}(\{y=-x\})\right)$$
that has zero Lebesgue measure.

We will have a region of interest for each coordinate, and it is defined by when $\vert i_0 \vert=1$ for the $i$-coordinate and $\vert j_0 \vert=1$ for the $j$-coordinate. This particular region has to be highlighted because it is exactly where the dynamics change. The symbol that will be attributed to the point in each instant and for each coordinate is:
\begin{eqnarray*}
i_0 \oplus_{n\in\mathbb{N}} i_n \stackrel{s_i}{\mapsto} \left\{
\begin{array}{rcl}
&2& \ i_0>1 \\
&1& \ i_0=1 \\
&-1& \ i_0=-1 \\
&-2& \ i_0<-2
\end{array}\right.
\end{eqnarray*}
and the same for the $j$-coordinate. Given any point $p$ with full orbit defined, the sequence we will associate is linked to the itinerary of the point:
$$(\dots\underbrace{s_{i_{-2}}}_{f^{-2}(p)} \underbrace{s_{i_{-1}}}_{f^{-1}(p)}\textbf{;} \underbrace{s_{i_{0}}}_{p} \underbrace{s_{i_1}}_{f(p)}\underbrace{s_{i_2}}_{f^2(p)}\dots \ ,\ \dots\underbrace{s_{j_{-2}}}_{f^{-2}(p)} \underbrace{s_{j_{-1}}}_{f^{-1}(p)}\textbf{;} \underbrace{s_{j_{0}}}_{p} \underbrace{s_{j_1}}_{f(p)}\underbrace{s_{j_2}}_{f^2(p)}\dots) $$
where $s_{i_n}$ represents the symbol that has to be associated to the $i$-coordinate at the instant $f^{i_n}(p)$ and $s_{j_m}$ represents the symbol that has to be associated to the $j$-coordinate at the instant $f^{j_m}(p)$.

\begin{example}
Let us take a moment to understand how the coding will take place with some examples.
{\footnotesize
\begin{table}[h]
\centering
\vspace{0.5cm}
\begin{tabular}{c|c|c|c|c}
 
& Initial Point $p$ & $f(p)$ & $f^2(p)$ \\ 
\hline \hline    \hline                         
coordinates & $(3\oplus -2, 1\oplus -4)$  & $(2\oplus -2, 2\oplus -4)$ & $(1\oplus -2, 3\oplus -4)$   \\ \hline
coding &  $(2 \ ,\ 1)$  &  $(22 \ , \ 12)$    &  $(221 \ , \ 122)$   \\ \hline\hline\hline

coordinates &  $(-1\oplus 3\oplus -1, -1)$   & $( 3\oplus -1, -2)$   &  $( 2\oplus -1, 1\oplus -2)$   \\ \hline
coding & $(-1 \ , \ -1)$  &  $(-12 \ , \ -1-2)$  &  $(\operatorname{-122} \ , \ \operatorname{-1-21})$  \\ \hline\hline\hline 

coordinates &  $(-1\oplus 1\oplus -1, -1)$   & $( 1\oplus -1, -2)$   &  $(  -1, 1\oplus -2)$   \\ \hline
coding & $(-1 \ , \ -1)$  &  $(\operatorname{-11} \ , \ \operatorname{-1-2})$  &  $(\operatorname{-11-1} \ , \ \operatorname{-1-21})$  \\ \hline\hline\hline 

coordinates &  $(3\oplus -2, 3\oplus -4)$  & $(2\oplus -2, 4\oplus -4)$ & $(1\oplus -2, 5\oplus -4)$   \\ \hline
coding & $(2 \ ,\ 2)$  &  $(22 \ , \ 22)$    &  $(221 \ , \ 222)$  \\ \hline\hline\hline 

coordinates &  $(3\oplus -2, -1\oplus 4)$  & $(2\oplus -2, 1\oplus-1\oplus -4)$ & $(1\oplus -2, 2\oplus-1\oplus -4)$   \\ \hline
coding & $(2 \ ,\ \operatorname{-1})$  &  $(22 \ , \ \operatorname{-11})$    &  $(221 \ , \ \operatorname{-112})$  \\ \hline\hline\hline 
\end{tabular}
\end{table}}

To completely understand how the orbits behave under the iteration of $f$, keep in mind the description we introduced before using the coordinates. It makes easier to see how $f$ acts in the coordinates and just compute which $R$ and $L$-stripe you are.
\end{example}

We will deal with each one of the coordinates separately, first defining the coding in the $i$-coordinate and then proving some lemmas about it. The $j$-coordinate will be dealt latter on but the the idea is pretty much the same. Even tough they ``see'' different things, they have an intrinsic relation that will become very clear once we clarify the coding.

\subsubsection{Coding the $i$-coordinate}

To code the $i$-coordinate, let us define for $p=(i_0 \oplus_{n\in\mathbb{N}} i_n, j_0 \oplus_{m\in\mathbb{N}} j_m)$, the definition of $h_i(p)$ is split depending on the sign of $i_0$ and $j_0$:
{\footnotesize
$$\dots \underbrace{\pm2 \dots \pm2 \pm1}_{\vert j_2 \vert}\underbrace{\mp2 \dots \mp2 \mp1}_{\vert j_1 \vert}\underbrace{\pm2 \dots \pm2}_{\vert j_0 \vert}\textbf{;}\underbrace{\pm2 \dots \pm2}_{\vert i_0 \vert-1} \underbrace{\pm1 \mp2 \dots \mp2}_{\vert i_1 \vert}\underbrace{\mp1 \pm2 \dots \pm2}_{\vert i_2 \vert}\dots$$}
if $\operatorname{sign}(i_0)=\operatorname{sign}(j_0)$ and
{\footnotesize
$$\dots \underbrace{\pm2 \dots \pm2 \pm1}_{\vert j_2 \vert}\underbrace{\mp2 \dots \mp2 \mp1}_{\vert j_1 \vert}\underbrace{\pm2 \dots \pm1}_{\vert j_0 \vert}\textbf{;}\underbrace{\mp2 \dots \mp2}_{\vert i_0 \vert-1} \underbrace{\mp1 \pm2 \dots \pm2}_{\vert i_1 \vert}\underbrace{\pm1 \mp2 \dots \mp2}_{\vert i_2 \vert}\dots$$}
if $\operatorname{sign}(i_0)\neq\operatorname{sign}(j_0)$, where the sign of each block is the same of the sign of $i_n$ and $j_m$. If any $j_m$ or $i_n$ has module 1, then the block associated to it will only be the respective 1, and if it has higher module you start ``adding'' 2's. To clarify the idea, let us check some examples

\begin{example}
Here we are going to code some examples just to help understand exactly how $h_i$ codes the $i$-coordinate.
\begin{itemize}
\item[(i)] $(3\oplus -2\oplus \dots, 1\oplus -4 \oplus \dots)$: The sign of $i_0$ and $j_0$ are equal then
$$h_i(3\oplus -2\oplus \dots, 1\oplus -4\oplus \dots)= \dots \underbrace{\operatorname{-2-2-2-1}}_{\vert -4 \vert}\underbrace{\operatorname{2}}_{\vert 1 \vert}; \underbrace{\operatorname{22}}_{\vert 3 \vert-1}\underbrace{\operatorname{1-2}}_{\vert -2 \vert}\dots$$

\item[(ii)] $(-1\oplus 3\oplus -1\dots, -1\oplus 2 \oplus\dots)$: Once again they have the same sign
$$h_i(-1\oplus 3\oplus -1\dots, -1\oplus 2 \oplus\dots)=\dots \underbrace{\operatorname{21}}_{\vert 2 \vert}\underbrace{\operatorname{-2}}_{\vert -1 \vert}; \underbrace{}_{\vert -1 \vert-1}\underbrace{\operatorname{-122}}_{\vert 3 \vert}\underbrace{\operatorname{1}}_{\vert -1 \vert}\dots$$

\item[(iii)] $(3\oplus -2\oplus \dots, -1\oplus 4\oplus\dots)$: Now $i_0$ and $j_0$ have different signs
$$h_i(3\oplus -2\oplus \dots, -1\oplus 4\oplus\dots)= \dots \underbrace{\operatorname{2221}}_{\vert 4 \vert}\underbrace{\operatorname{-1}}_{\vert -1 \vert}; \underbrace{\operatorname{22}}_{\vert 3 \vert-1}\underbrace{\operatorname{1-2}}_{\vert -2 \vert}\dots$$
\end{itemize}
\end{example}

Let $\sigma_i:\Sigma \rightarrow \Sigma$ be the shift map on the space of the sequences over the alphabet $\mathcal{A}$. Then

\begin{lemma}
$\sigma_i\circ h_i = h_i \circ f$
\end{lemma}

\begin{proof}
We will do the proof only looking at the upper plane of $\mathbb{R}^2$ due the symmetry of $f$. Hence
\begin{itemize}
\item $i_0>0 \quad j_0>0$:

Let $p=(i_0 \oplus_{n\in\mathbb{N}} i_n, j_0 \oplus_{m\in\mathbb{N}} j_m)$, then we know that
\begin{eqnarray*}
f(p)=\left\{\begin{array}{rcl}
&\left(i_0 - 1 \oplus_{n\in\mathbb{N}} i_n, j_0 + 1\oplus_{m\in\mathbb{N}} j_m\right)& \ i_0>1; \\
& \left(i_1 \oplus_{n\geq 2} i_n, j_0 + 1\oplus_{m\in\mathbb{N}} j_m\right)& \ i_0=1; 
\end{array}\right.
\end{eqnarray*}
which implies that
\begin{eqnarray*}
h_i\circ f(p)=\left\{\begin{array}{rcl}
& \dots \underbrace{\operatorname{-2\dots -2-1}}_{\vert j_1 \vert}\underbrace{\operatorname{2\dots 2}}_{\vert j_0 +1\vert}; \underbrace{\operatorname{2\dots2}}_{\vert i_0 - 1 \vert-1}\underbrace{\operatorname{1-2\dots -2}}_{\vert i_1 \vert}\dots  & \ i_0>1; \\

&\dots \underbrace{\operatorname{-2\dots -2-1}}_{\vert j_1 \vert}\underbrace{\operatorname{2\dots 21}}_{\vert j_0 +1\vert}; \underbrace{\operatorname{-2\dots -2}}_{\vert i_1 \vert-1}\underbrace{\operatorname{-12\dots 2}}_{\vert i_2 \vert}\dots& \ i_0=1; 
\end{array}\right.
\end{eqnarray*}
once we have alternating signs for $i_0$ and $i_1$. Also, we know that
\begin{eqnarray*}
h_i(p)=\left\{\begin{array}{rcl}
& \dots \underbrace{\operatorname{-2\dots -2-1}}_{\vert j_1 \vert}\underbrace{\operatorname{2\dots 2}}_{\vert j_0\vert}; \underbrace{\operatorname{2\dots2}}_{\vert i_0 \vert-1}\underbrace{\operatorname{1-2\dots -2}}_{\vert i_1 \vert}\dots  & \ i_0>1; \\

&\dots \underbrace{\operatorname{-2\dots -2-1}}_{\vert j_1 \vert}\underbrace{\operatorname{2\dots 2}}_{\vert j_0\vert};\underbrace{\operatorname{1-2\dots -2}}_{\vert i_1 \vert}\underbrace{\operatorname{-12\dots 2}}_{\vert i_2 \vert}\dots& \ i_0=1; 
\end{array}\right.
\end{eqnarray*}

Now applying the shift we get
\begin{eqnarray*}
\sigma_i \circ h_i(p)=\left\{\begin{array}{rcl}
& \dots \underbrace{\operatorname{-2\dots -2-1}}_{\vert j_1 \vert}\underbrace{\operatorname{2\dots 2}}_{\vert j_0 \vert +1}; \underbrace{\operatorname{2\dots2}}_{\vert i_0 \vert-2}\underbrace{\operatorname{1-2\dots -2}}_{\vert i_1 \vert}\dots  & \ i_0>1; \\

&\dots \underbrace{\operatorname{-2\dots -2-1}}_{\vert j_1 \vert}\underbrace{\operatorname{2\dots 21}}_{\vert j_0 \vert +1}; \underbrace{\operatorname{-2\dots -2}}_{\vert i_1 \vert-1}\underbrace{\operatorname{-12\dots 2}}_{\vert i_2 \vert}\dots& \ i_0=1; 
\end{array}\right.
\end{eqnarray*}
and proves the statement in these cases.

\item $i_0>0 \quad j_0<0$:

The proof here is basically the same, we only change how we apply $f$
\begin{eqnarray*}
f\left(p\right)=\left\{\begin{array}{rcl}
&\left(i_0 - 1 \oplus_{n\in\mathbb{N}} i_n, 1\oplus j_0 \oplus_{m\in\mathbb{N}} j_m\right)& \ i_0>1 \\
& \left(i_1 \oplus_{n\geq 2} i_n, 1 \oplus j_0 \oplus_{m\in\mathbb{N}} j_m\right)& \ i_0=1;
\end{array}\right.
\end{eqnarray*}
and then
\begin{eqnarray*}
h_i\circ f(p)=\left\{\begin{array}{rcl}
& \dots \underbrace{\operatorname{2\dots 21}}_{\vert j_1 \vert}\underbrace{\operatorname{-2\dots -2-1}}_{\vert j_0\vert} \underbrace{2}_1 ; \underbrace{\operatorname{2\dots2}}_{\vert i_0 - 1 \vert-1}\underbrace{\operatorname{1-2\dots -2}}_{\vert i_1 \vert}\dots  & \ i_0>1; \\

&\dots \underbrace{\operatorname{2\dots 21}}_{\vert j_1 \vert}\underbrace{\operatorname{-2\dots -2-1}}_{\vert j_0\vert}\underbrace{1}_1; \underbrace{\operatorname{-2\dots -2}}_{\vert i_1 \vert-1}\underbrace{\operatorname{-12\dots 2}}_{\vert i_2 \vert}\dots& \ i_0=1; 
\end{array}\right.
\end{eqnarray*}

Also we have that
\begin{eqnarray*}
h_i(p)=\left\{\begin{array}{rcl}
& \dots \underbrace{\operatorname{2\dots 21}}_{\vert j_1 \vert}\underbrace{\operatorname{-2\dots -2-1}}_{\vert j_0\vert}; \underbrace{\operatorname{2\dots2}}_{\vert i_0 \vert-1}\underbrace{\operatorname{1-2\dots -2}}_{\vert i_1 \vert}\dots  & \ i_0>1; \\

&\dots \underbrace{\operatorname{2\dots 21}}_{\vert j_1 \vert}\underbrace{\operatorname{-2\dots -2-1}}_{\vert j_0\vert};\underbrace{\operatorname{1-2\dots -2}}_{\vert i_1 \vert}\underbrace{\operatorname{-12\dots 2}}_{\vert i_2 \vert}\dots& \ i_0=1; 
\end{array}\right.
\end{eqnarray*}
and applying the shift
\begin{eqnarray*}
\sigma_i \circ h_i(p)=\left\{\begin{array}{rcl}
& \dots \underbrace{\operatorname{2\dots 21}}_{\vert j_1 \vert}\underbrace{\operatorname{-2\dots -2-1}}_{\vert j_0\vert}\underbrace{\operatorname{2}}_{1}; \underbrace{\operatorname{2\dots2}}_{\vert i_0 \vert-2}\underbrace{\operatorname{1-2\dots -2}}_{\vert i_1 \vert}\dots  & \ i_0>1; \\

&\dots \underbrace{\operatorname{2\dots 21}}_{\vert j_1 \vert}\underbrace{\operatorname{-2\dots -2-1}}_{\vert j_0\vert}\underbrace{\operatorname{1}}_{1};\underbrace{\operatorname{-2\dots -2}}_{\vert i_1 \vert-1}\underbrace{\operatorname{-12\dots 2}}_{\vert i_2 \vert}\dots& \ i_0=1; 
\end{array}\right.
\end{eqnarray*}
\end{itemize}

Therefore putting together both items above, we conclude the Lemma's proof.
\end{proof}

The last thing regarding the $i$-coordinate coding that is needed to be discussed is how to code the point that have finite orbit foreword or backward. As we discussed before, this happens if you are on a pre-image of $\{y=0\}$ or on an image of the $\{y=-x\}$, which implies that you have a finite $i$ of $j$-coordinate. In this case, you just use the $h_i$ defined and when you reach the ``final'' number you just put 0 in the next step and cease to code. These point will have a finite coding backward or forward.

\begin{example}
Each one of the examples below has different type of finite orbit. To fully understand what is happening here, try to visualise the geometric interpretation of the finite orbit.
\begin{itemize}
\item[(i)] $(3, 1\oplus -4 \oplus \dots)$: This point lies over the $f^{-3}(\{y=0\})$ but it is not over any image of $\{y=-x\}$
$$h_i(3, 1\oplus -4\oplus \dots)= \dots \underbrace{\operatorname{-2-2-2-1}}_{\vert -4 \vert}\underbrace{\operatorname{2}}_{\vert 1 \vert}; \underbrace{\operatorname{22}}_{\vert 3 \vert-1}0$$

\item[(ii)] $(3\oplus -2\oplus \dots, 1\oplus -4)$: The point here is over $f^4(\{y=-x\})$ but it is not over any pre-images of $\{y=0\}$
$$h_i(3\oplus -2\oplus \dots, 1\oplus -4)= 0\underbrace{\operatorname{-2-2-2-1}}_{\vert -4 \vert}\underbrace{\operatorname{2}}_{\vert 1 \vert}; \underbrace{\operatorname{22}}_{\vert 3 \vert-1}\underbrace{\operatorname{1-2}}_{\vert -2 \vert}\dots$$

\item[(iii)] $(3, 1\oplus -4)$: This one lies over one of the intersections between $f^{-3}(\{y=0\})$ and $f^4(\{y=-x\})$
$$h_i(3, 1\oplus -4)= 0\underbrace{\operatorname{-2-2-2-1}}_{\vert -4 \vert}\underbrace{\operatorname{2}}_{\vert 1 \vert}; \underbrace{\operatorname{22}}_{\vert 3 \vert-1}0$$
\end{itemize}
\end{example}

\subsubsection{Coding the $j$-coordinate}

The $j$-coordinate will have the same kind of coding and, in fact, it is possible to see a direct relation between both coordinates. They have an strict relation and it will become very clear once we define the other map.

To code the $j$-coordinate, let  $p=(i_0 \oplus_{n\in\mathbb{N}} i_n, j_0 \oplus_{m\in\mathbb{N}} j_m)$ be a point with full orbit, the definition of $h_j(p)$ is once again split depending on the sign of $i_0$ and $j_0$:
{\tiny
$$\dots \underbrace{\pm1 \pm2 \dots \pm2}_{\vert j_2 \vert}\underbrace{\mp1 \mp2 \dots \mp2}_{\vert j_1 \vert}\underbrace{\pm1 \pm2 \dots \textbf{;} \pm2}_{\vert j_0 \vert}\underbrace{\pm2 \dots \pm2}_{\vert i_0 \vert} \underbrace{\mp1 \mp2 \dots \mp2}_{\vert i_1 \vert}\underbrace{\pm1 \pm2 \dots \pm2}_{\vert i_2 \vert}\dots$$}
if $\operatorname{sign}(i_0)=\operatorname{sign}(j_0)$ and
{\tiny
$$\dots \underbrace{\pm1 \pm2 \dots \pm2}_{\vert j_2 \vert}\underbrace{\mp1 \mp2 \dots \mp2}_{\vert j_1 \vert}\underbrace{\pm1 \pm2 \dots \textbf{;} \pm2}_{\vert j_0 \vert}\underbrace{\pm 1 \pm2 \dots \pm2}_{\vert i_0 \vert} \underbrace{\mp1 \mp2 \dots \mp2}_{\vert i_1 \vert}\underbrace{\pm1 \pm2 \dots \pm2}_{\vert i_2 \vert}\dots$$}
if $\operatorname{sign}(i_0)\neq\operatorname{sign}(j_0)$, where the sign of each block is the same of the sign of $i_n$ and $j_m$.Keep in mind that if any $j_m$ or $i_n$ has module 1, then the block associated to it will only be the respective 1. Lets look once more to the examples we presented before, but now under the $j$-perspective:

\begin{example}
Here we are going to code some examples just to help understand exactly how $h_j$ codes the $j$-coordinate.
\begin{itemize}
\item[(i)] $(3\oplus -2\oplus \dots, 1\oplus -4 \oplus \dots)$: The sign of $i_0$ and $j_0$ are equal then
$$h_j(3\oplus -2\oplus \dots, 1\oplus -4\oplus \dots)= \dots \underbrace{\operatorname{-1-2-2-2}}_{\vert -4 \vert}\textbf{;}\underbrace{\operatorname{1}}_{\vert 1 \vert} \underbrace{\operatorname{222}}_{\vert 3 \vert}\underbrace{\operatorname{-1-2}}_{\vert -2 \vert}\dots$$

\item[(ii)] $(-1\oplus 3\oplus -1\dots, -1\oplus 2 \oplus\dots)$: Once again they have the same sign
$$h_j(-1\oplus 3\oplus -1\dots, -1\oplus 2 \oplus\dots)=\dots \underbrace{\operatorname{12}}_{\vert 2 \vert}\textbf{;}\underbrace{\operatorname{-1}}_{\vert -1 \vert} \underbrace{-2}_{\vert -1 \vert}\underbrace{\operatorname{122}}_{\vert 3 \vert}\underbrace{\operatorname{-1}}_{\vert -1 \vert}\dots$$

\item[(iii)] $(3\oplus -2\oplus \dots, -1\oplus 4\oplus\dots)$: Now $i_0$ and $j_0$ have different signs
$$h_j(3\oplus -2\oplus \dots, -1\oplus 4\oplus\dots)= \dots \underbrace{\operatorname{1222}}_{\vert 4 \vert}\textbf{;}\underbrace{\operatorname{-1}}_{\vert -1 \vert} \underbrace{\operatorname{122}}_{\vert 3 \vert-1}\underbrace{\operatorname{-1-2}}_{\vert -2 \vert}\dots$$
\end{itemize}
\end{example}

Let $\sigma_j:\Sigma \rightarrow \Sigma$ be the shift map on the space of the sequences over the alphabet $\mathcal{A}$. Then

\begin{lemma}
$\sigma_j\circ h_j = h_j \circ f$
\end{lemma}

\begin{proof}
We will do the proof only looking at the upper plane of $\mathbb{R}^2$ due the symmetry of $f$. Hence
\begin{itemize}
\item $i_0>0 \quad j_0>0$:

Let $p=(i_0 \oplus_{n\in\mathbb{N}} i_n, j_0 \oplus_{m\in\mathbb{N}} j_m)$, then we know that
\begin{eqnarray*}
f(p)=\left\{\begin{array}{rcl}
&\left(i_0 - 1 \oplus_{n\in\mathbb{N}} i_n, j_0 + 1\oplus_{m\in\mathbb{N}} j_m\right)& \ i_0>1; \\
& \left(i_1 \oplus_{n\geq 2} i_n, j_0 + 1\oplus_{m\in\mathbb{N}} j_m\right)& \ i_0=1; 
\end{array}\right.
\end{eqnarray*}
which implies that
\begin{eqnarray*}
h_j\circ f(p)=\left\{\begin{array}{rcl}
& \dots \underbrace{\operatorname{-1-2\dots -2}}_{\vert j_1 \vert}\underbrace{\operatorname{12\dots \textbf{;}2}}_{\vert j_0 +1\vert} \underbrace{\operatorname{2\dots2}}_{\vert i_0 - 1 \vert}\underbrace{\operatorname{-1-2\dots -2}}_{\vert i_1 \vert}\dots  & \ i_0>1; \\

&\dots \underbrace{\operatorname{-1-2\dots -2}}_{\vert j_1 \vert}\underbrace{\operatorname{12\dots 2}}_{\vert j_0 +1\vert}\textbf{;} \underbrace{\operatorname{-1-2\dots -2}}_{\vert i_1 \vert}\underbrace{\operatorname{12\dots 2}}_{\vert i_2 \vert}\dots& \ i_0=1; 
\end{array}\right.
\end{eqnarray*}
once we have alternating signs for $i_0$ and $i_1$. Also, we know that
\begin{eqnarray*}
h_j(p)=\left\{\begin{array}{rcl}
& \dots \underbrace{\operatorname{-1-2\dots -2}}_{\vert j_1 \vert}\underbrace{\operatorname{12\dots \textbf{;}2}}_{\vert j_0\vert} \underbrace{\operatorname{2\dots2}}_{\vert i_0 \vert}\underbrace{\operatorname{-1-2\dots -2}}_{\vert i_1 \vert}\dots  & \ i_0>1; \\

&\dots \underbrace{\operatorname{-1-2\dots -2}}_{\vert j_1 \vert}\underbrace{\operatorname{12\dots \textbf{;}2}}_{\vert j_0\vert}\underbrace{2}_{\vert i_0\vert}\underbrace{\operatorname{-1-2\dots -2}}_{\vert i_1 \vert}\underbrace{\operatorname{12\dots 2}}_{\vert i_2 \vert}\dots& \ i_0=1; 
\end{array}\right.
\end{eqnarray*}

Now applying the shift we get
\begin{eqnarray*}
\sigma_j \circ h_j(p)=\left\{\begin{array}{rcl}
& \dots \underbrace{\operatorname{-1-2\dots -2}}_{\vert j_1 \vert}\underbrace{\operatorname{12\dots \textbf{;}2}}_{\vert j_0 \vert +1} \underbrace{\operatorname{2\dots2}}_{\vert i_0 \vert-1}\underbrace{\operatorname{-1-2\dots -2}}_{\vert i_1 \vert}\dots  & \ i_0>1; \\

&\dots \underbrace{\operatorname{-1-2\dots -2}}_{\vert j_1 \vert}\underbrace{\operatorname{12\dots \textbf{;}2}}_{\vert j_0 \vert +1} \underbrace{\operatorname{-1-2\dots -2}}_{\vert i_1 \vert-1}\underbrace{\operatorname{12\dots 2}}_{\vert i_2 \vert}\dots& \ i_0=1; 
\end{array}\right.
\end{eqnarray*}
and proves the statement in these cases.

\item $i_0>0 \quad j_0<0$:

The proof here is basically the same, we only change how we apply $f$
\begin{eqnarray*}
f\left(p\right)=\left\{\begin{array}{rcl}
&\left(i_0 - 1 \oplus_{n\in\mathbb{N}} i_n, 1\oplus j_0 \oplus_{m\in\mathbb{N}} j_m\right)& \ i_0>1 \\
& \left(i_1 \oplus_{n\geq 2} i_n, 1 \oplus j_0 \oplus_{m\in\mathbb{N}} j_m\right)& \ i_0=1;
\end{array}\right.
\end{eqnarray*}
and then
\begin{eqnarray*}
h_j\circ f(p)=\left\{\begin{array}{rcl}
& \dots \underbrace{\operatorname{12\dots 2}}_{\vert j_1 \vert}\underbrace{\operatorname{-1-2\dots -2}}_{\vert j_0\vert} \textbf{;}\underbrace{1}_1  \underbrace{\operatorname{2\dots2}}_{\vert i_0 -1 \vert}\underbrace{\operatorname{-1-2\dots -2}}_{\vert i_1 \vert}\dots  & \ i_0>1; \\

&\dots \underbrace{\operatorname{12\dots 2}}_{\vert j_1 \vert}\underbrace{\operatorname{-1-2\dots -2}}_{\vert j_0\vert}\textbf{;}\underbrace{1}_1 \underbrace{\operatorname{-1-2\dots -2}}_{\vert i_1 \vert}\underbrace{\operatorname{12\dots 2}}_{\vert i_2 \vert}\dots& \ i_0=1; 
\end{array}\right.
\end{eqnarray*}

Also we have that
\begin{eqnarray*}
h_j(p)=\left\{\begin{array}{rcl}
& \dots \underbrace{\operatorname{12\dots 2}}_{\vert j_1 \vert}\underbrace{\operatorname{-1-2\dots \textbf{;}-2}}_{\vert j_0\vert} \underbrace{\operatorname{12\dots2}}_{\vert i_0 \vert}\underbrace{\operatorname{-1-2\dots -2}}_{\vert i_1 \vert}\dots  & \ i_0>1; \\

&\dots \underbrace{\operatorname{12\dots 2}}_{\vert j_1 \vert}\underbrace{\operatorname{-1-2\dots \textbf{;}-2}}_{\vert j_0\vert}\underbrace{\operatorname{1}}_{\vert i_0 \vert}\underbrace{\operatorname{-1-2\dots -2}}_{\vert i_1 \vert}\underbrace{\operatorname{112\dots 2}}_{\vert i_2 \vert}\dots& \ i_0=1; 
\end{array}\right.
\end{eqnarray*}
and applying the shift
\begin{eqnarray*}
\sigma_j \circ h_j(p)=\left\{\begin{array}{rcl}
& \dots \underbrace{\operatorname{12\dots 2}}_{\vert j_1 \vert}\underbrace{\operatorname{-1-2\dots -2}}_{\vert j_0\vert}\textbf{;}\underbrace{\operatorname{1}}_{1} \underbrace{\operatorname{2\dots2}}_{\vert i_0 \vert-1}\underbrace{\operatorname{-1-2\dots -2}}_{\vert i_1 \vert}\dots  & \ i_0>1; \\

&\dots \underbrace{\operatorname{12\dots 2}}_{\vert j_1 \vert}\underbrace{\operatorname{-1-2\dots -2}}_{\vert j_0\vert}\textbf{;}\underbrace{\operatorname{1}}_{1}\underbrace{\operatorname{-1-2\dots -2}}_{\vert i_1 \vert}\underbrace{\operatorname{12\dots 2}}_{\vert i_2 \vert}\dots& \ i_0=1; 
\end{array}\right.
\end{eqnarray*}
\end{itemize}
therefore putting together both items above, we conclude the Lemma's proof.
\end{proof}

As before, we define here the image of the point with finite orbit in the exact same way as before: just add zero after using all the available $i_n$'s and $j_m$'s.


\subsection{Conjugacy and its consequences}

Each one of the coordinates identifies every time the point enters the zone of interest and how long it takes to get there. The length of each block between each $\pm1$ is how long it will take to return the region delimited by $R_1$ and $R_{-1}$ in the $i$-coordinate and $L_1$ and $L_{-1}$ in the $j$-coordinate. 

We can restate the theorem by being a bit more precise about each one of the subshifts we mentioned and also the precise map

\begin{theorem*}[{\bf \ref{teoB}'}]
Let $\Sigma_i$ and $\Sigma_j$ be the image of $h_i(\mathbb{R}^2)$ and $h_j(\mathbb{R}^2)$, respectively. Define the map
$$\begin{array}{rcl}
h:\mathbb{R}^2 &\rightarrow& \Sigma_i \times \Sigma_j \\
p &\mapsto& (h_i(p),h_j(p))
\end{array}
$$
and the following diagram commutes
\begin{displaymath}
\xymatrix{
{\mathbb{R}^2 } \ar[r]^{f} \ar[d]_{h} & {\mathbb{R}^2 \ar[d]^{h}} \\
{\Sigma_i \times \Sigma_j} \ar[r]_{\sigma}  & {\Sigma_i \times \Sigma_j}}
\end{displaymath}
where 
$$\begin{array}{rcl}
\sigma:\Sigma_i \times \Sigma_j &\rightarrow& \Sigma_i \times \Sigma_j \\
((s_n),(s_m)) &\mapsto& (\sigma_i(s_n),\sigma_j(s_m))
\end{array}
$$
\end{theorem*}

\vspace{.3cm}

One can be precise when restricted to the points where the full orbit is defined, that is, the set of points in $\mathbb{R}^2$ which orbit never meets the discontinuities: 
{\footnotesize
$$\Sigma_k' = h_k\left( \mathbb{R}^2 \setminus \left(\left(\bigcup_{\{n\in\mathbb{N}\}}f^{-n}(\{y=0\}\right)\bigcup \left(\bigcup_{\{n\in\mathbb{N}\}}f^{n}(\{y=-x\}\right) \right) \right)$$}
for $k=i,j$. This tells us that each $\Sigma_i', \Sigma_j' \subset \{-2,-1,1,2\}^{\mathbb{Z}}$ and that  $h$ is continuous restricted to this set.

Which gives this immediate consequence.

\begin{corollary*}
The H\'{e}non-Devaney map has a density of hyperbolic periodic points.
\end{corollary*}

The coding we introduce here can be seen as a two-dimensional version of the one for the Boole's map. The construction follows the same idea as before, that is, looking at the pre-images of the discontinuity. 

We introduce a new system of coordinates in the real line, just like the $i$-coordinate. Let $\{R_i\}_{i\geq0}$ be the points given the highest $i$-pre-image of $\{x=0\}$. The mirrored points will be denoted by  $\{R_i\}_{i<0}$. Each interval $(R_{i-1},R_{i})$ will only be denoted as $i$.

Understanding how the pre-images spread throughout the real line is a fundamental part of the results presented in \cite{Adler}, therefore the analogous of Lemma \ref{Lemma1} is already known. Actually here is the main difference between the statements because here we do have the density of pre-images. Hence, the map $h$ we are going to construct is indeed a bijection. 

The points inside $i=1$ that are induced by the pre-images of the $\{R_i\}_{i<0}$, that is, it is a family of points $R_{1\oplus i_1}$ contained in the interval $i=1$ such that $B(R_{1\oplus i_1})\subset R_{i_1}, i_1\in\mathbb{Z}^-$, where $R_{1\oplus i_1} = B^{-1}(R_{i_1})$.

Proceeding in the same way as before, we get a set of points in the real line given by 
$$\{R_{i_0 \oplus_{n\in\mathbb{N}} i_n}; \mbox{ where } i_n\in \mathbb{Z} \mbox{ of alternating signs}\}$$
and the Boole's map can be seen as
{
\begin{eqnarray*}
B\left(i_0 \oplus_{n\in\mathbb{N}} i_n \right)=\left\{\begin{array}{rcl}
&\left(i_0 - 1 \oplus_{n\in\mathbb{N}} i_n \right)& \ i_0>1; \\
& \left(i_1 \oplus_{n\geq 2} i_n\right)& \ i_0=1; 
\end{array}\right.
\end{eqnarray*}}

The symbol $1$ is given to a point that is in $(0,1)$, $2$ to the points in $(1,\infty)$. Analogously to -$1$ and $-2$. The $0$ is given once again to the points that is one of the pre-images of zero.

For any point $p\in\mathbb{R}$, the positive sequence associated to it is given by
$$(\underbrace{s_{i_{0}}}_{p} \underbrace{s_{i_1}}_{f(p)}\underbrace{s_{i_2}}_{f^2(p)}\dots) $$
where $s_{i_n}$ represents the symbol that has to be associated to the $i$-coordinate at the instant $B^{i_n}(p)$.

The process we explained in the previous section allows us to give get a similar coding to the classical Boole, that looks like a restriction of the Hénon-Devaney and it is only defined for positive time.

\begin{corollary*}
There exists $\Sigma_i \subset \Sigma$ and a bijective map $h:\mathbb{R} \rightarrow \Sigma_i$ such that the following diagram commutes
\begin{displaymath}
\xymatrix{
{\mathbb{R}} \ar[r]^{B} \ar[d]_{h} & {\mathbb{R} \ar[d]^{h}} \\
{\Sigma_i} \ar[r]_{\sigma}  & {\Sigma_i}}
\end{displaymath}
where $\sigma:\Sigma_i\rightarrow \Sigma_i$ is the usual shift and
{\footnotesize
$$\Sigma := \bigcup_{n \in \overline{\mathbb{N}}\cup\{0\}} \{(s_0,s_1,\dots,s_n) ; s_i\in\{ -2,-1,1,2\}, s_n=0\}$$}
\end{corollary*}

Our initial goal with this coding was trying to get some tools walking towards the recurrence of the H\'{e}non-Devaney map, however we managed to get something a bit weaker than that. With this coding we can only get ``density of recurrence'' in the sense that, given an open set $R$ in the plane we can find a dense of orbits that enters $R$ in finite time, even more, we can determine in which time we want the point enters the region. As a consequence of this fact, we also get that there exists a orbit which is dense in the plane. Some more information will be given in the upcoming paper \cite{Pujals2}, a joint work with Ernique Pujals.



\end{document}